\documentclass[11pt]{amsart}
\usepackage{amssymb, adjustbox, enumerate, amsbsy, stmaryrd}
\usepackage{geometry}
\usepackage{todonotes}

\usepackage{amsfonts, amssymb, amscd}
\numberwithin{equation}{section}

\usepackage[symbol]{footmisc}

\usepackage{bm}
\usepackage{verbatim}
\usepackage{mathrsfs}
\usepackage{graphicx}
\usepackage{tikz-cd}
\usepackage{subcaption}
\usepackage{listings}
\usepackage{subfiles}
\usepackage[toc,page]{appendix}
\usepackage{mathtools}
\usepackage{comment}
\usepackage{enumerate}
\usepackage{enumitem}
\usepackage[all]{xy}

\usepackage{graphicx}
\graphicspath{{images/}}

\usepackage{appendix}
\usepackage{hyperref}
\hypersetup{
    colorlinks=true,
    citecolor=red,
    linkcolor=blue,
    filecolor=magenta,      
    urlcolor=red,
}
\newcommand{\bb}{\bm{b}}
\newcommand{\Mm}{{\bf{M}}}
\newcommand{\Bb}{{\bf{B}}}
\newcommand{\PP}{{\bf{P}}}

\newcommand{\Dd}{{\bf{D}}}

\newcommand{\Spec}{\mathrm{Spec}}
\newcommand{\Src}{\mathrm{Src}}
\newcommand{\Spr}{\mathrm{Spr}}

\newcommand{\Cc}{\mathbb{C}}

\newcommand{\Pp}{\mathbb{P}}
\newcommand{\Qq}{\mathbb{Q}}

\newcommand{\Rr}{\mathbb{R}}

\newcommand{\Zz}{\mathbb{Z}}

\newcommand{\Center}{\operatorname{center}}

\newcommand{\Exc}{\operatorname{Exc}}

\newcommand{\Bir}{\operatorname{Bir}}

\newcommand{\proj}{\operatorname{Proj}}

\newcommand{\Supp}{\operatorname{Supp}}
\newcommand{\Ngklt}{\operatorname{Ngklt}}
\newcommand{\Nlc}{\operatorname{Nlc}}

\newcommand{\mult}{\operatorname{mult}}

\newcommand{\cont}{\operatorname{cont}}
\newcommand{\Gal}{\operatorname{Gal}}

\newcommand{\lf}{\lfloor}
\newcommand{\rf}{\rfloor}

\newcommand{\Oo}{\mathcal{O}}

\newcommand{\NE}{\mathrm{NE}}

\newcommand{\Pic}{\mathrm{Pic}}

\newtheorem{thm}{Theorem}[section]

\newtheorem{lem}[thm]{Lemma}

\theoremstyle{definition}
\newtheorem{defn}[thm]{Definition}

\theoremstyle{definition}

\newtheorem{defthm}[thm]{Definition-Theorem}

\newtheorem{ex}[thm]{Example}

\theoremstyle{definition}

\begin{document}

\title{Contraction theorem for generalized pairs}
\author{Lingyao Xie}


\address{Department of Mathematics, The University of Utah, Salt Lake City, UT 84112, USA}
\email{lingyao@math.utah.edu}

\date{\today}

\begin{abstract}
We use Koll\'ar's gluing theory to prove the contraction theorem for generalized pairs.
\end{abstract}

\maketitle
\tableofcontents

\section{Introduction}

We work over the field of complex numbers $\mathbb C$. 

In recent years, it has become increasingly clear that it is important to generalize results from the MMP for pairs to the MMP for generalized pairs, see \cite{Bir21} and references therein. One of the most important conjectures in the MMP is the abundance conjecture. It is expected that if $(X,D)$ is an lc pair and $K_X+D$ is nef then $K_X+D$ is semi-ample. An important result in this direction is \cite{FG14,HX16}, where it is shown that log abundant nef lc pairs are semi-ample. Unluckily, this is false for generalized lc pairs even if we assume log abundance (\cite[Example 1.4]{LX22a}). Nonetheless, some weaker semi-ampleness results related to the MMP are still believed to be true for generalized pairs and will lead to many interesting applications. Therefore it is important to understand exactly where the semi-ampleness starts to fail and what assumptions one should add to avoid this failure.



As in the log canonical case (cf. \cite[Section 5.5]{Kol13}), a generalized log canonical structure gives a stratification (called glc stratification) of a variety with respect to its glc centers (\cite[Section 4]{LX22b}). Thanks to the $\Pp^1$-link techniques developed in \cite[Theorem 1.4]{FS20}(cf. \cite[Theorem 3.5]{Bir20}), the glc stratification turns out to be nice and useful. For instance, we use the glc stratification to prove that any glc singularity is Du Bois (\cite[Section 6]{LX22b}). More importantly, we can do adjunction to glc centers via the generalized canonical bundle formula developed in \cite{Fil20,HL21b,JLX22}, thus this stratification allows us to use Koll\'ar's gluing theory to prove semi-ampleness properties by induction on the dimension. The essential difficulty is to show that some induced equivalence relation is finite. In the lc pair case, as explained in \cite{HX13,HX16}, these relations arise from some subgroup of $\Bir(V,\Delta_V)$, where $V$ is some hereditary lc center (\cite[Definition 5.30]{Kol13}), hence the finiteness of \Bb-representations will imply the required finiteness of relations. In the glc pair case, this finiteness of \Bb-representations can fail in general, which is the main reason that the abundance conjecture is not true for generalized pairs. However, there are certain situations where we can actually show the finiteness of Koll\'ar's relations regardless of the \Bb-representations. In these cases, we expect the corresponding semi-ampleness results hold.
For example, in \cite{LX22b} Jihao Liu and the author prove the existence of glc flips and an analogue of \cite[Theorem 1.1]{HX13} in the setting of generalized pairs. 


Except for the semi-ampleness, the Minimal Model Program seems to work pretty well for generalized pairs. For example, termination of flips and existence of minimal models or Mori fiber spaces hold for many generalized pairs with some standard assumptions similar to the case of usual lc pairs (\cite[Theorem 4.1]{HL22},\cite[Theorem 1.1]{LT22},\cite[Theorem 3.17]{Has22},\cite[Theorem 1.2, 1.3]{LX22a}). As summarized in \cite{HL21a}, many very general results concerning running the MMP for lc/dlt pairs are still true for glc/gdlt pairs. Moreover, \cite{HL21a} established the Cone theorem for glc pairs, and also established the Contraction theorem when $\Mm_X$ is $\Rr$-Cartier. Their approach involves replacing the generalized pairs by some auxiliary usual pairs (\cite[Theorem 4.1]{HL21a}) with the help of some ample divisor. However, there are essential differences between glc pairs and the lc pairs when the ambient variety is not $\Qq$-factorial (\cite[Example 2.1]{LX22b}). Hence in order to obtain the contraction theorem for glc pairs in full generality, which is equivalent to showing some semi-ampleness, we have to extend the theory for generalized pairs instead of just using theorems developed for lc pairs.

The main purpose of this paper is to use the glc stratification developed in \cite{LX22b} and Koll\'ar's gluing theory to prove the following semi-ampleness theorem: 

\begin{thm}[$=$Theorem \ref{thm: semi-ampleness for glc crepant log structure with log bigness}]\label{thm: main thm 1.1}
Let $(X,\Delta,\Mm)/U$ be a glc $\Qq$-pair, and $L$ a nef $\Qq$-divisor such that $L-(K_X+\Delta+\Mm_X)$ is nef and log big$/U$ with respect to $(X,\Delta,\Mm)$. Then $L$ is semi-ample over $U$. 
\end{thm}

Since ample divisors are automatically nef and log big, therefore in particular we have: 

\begin{thm}\label{thm: K_X+B+M+A is semi-ample for A ample}
Let $(X,\Delta,\Mm)/U$ be a glc pair and $A$ an ample$/U$ $\Rr$-divisor. Then $K_X+\Delta+\Mm_X+A$ is nef$/U$ if and only if it is semi-ample$/U$. 
\end{thm}

As an easy corollary, we have:
\begin{thm}\label{thm: semi-ampleness with good B+(A)}
Let $(X,\Delta+A,\Mm)/U$ be a glc pair such that 
\begin{itemize}
    \item $\mathbf{B}_+(A/U)$ contains no glc center of $(X,\Delta+A,\Mm)$.
    \item $K_X+\Delta+A+\Mm_X$ is nef over $U$.
\end{itemize}
Then $K_X+\Delta+A+\Mm_X$ is semi-ample over $U$.
\end{thm}

By looking at the gluing relations more carefully, we can actually get a stronger result, which is the g-pair analogue of the base point free theorem for lc pairs.

\begin{thm}[Base point free theorem for glc pairs]\label{thm: bpf thm for g-pairs}
Let $(X,\Delta,\Mm)/U$ be a glc $\Qq$-pair, and $L$ a nef$/U$ Cartier divisor such that $aL-(K_X+\Delta+\Mm_X)$ is nef and log big$/U$ with respect to $(X,\Delta,\Mm)$ for some real positive number $a$. Then $\Oo_X(mL)$ is globally generated over $U$ for all $m\gg 0$. 
\end{thm}

The immediate application is the Contraction theorem for g-pairs, which fulfill the last part of \cite[Theorem 1.3]{HL21a} when $\Mm_X$ is not necessarily $\Rr$-Cartier:

\begin{thm}\label{thm: Contraction thm for g-pairs}
Let $(X,\Delta,\Mm)/U$ be a glc pair and $R$ be a $(K_X+\Delta+\Mm_X)$-negative extremal ray in
$\overline{\NE}(X/U)$. Then $R$ is a rational extremal ray. In particular, there exists a projective
morphism $\cont_R:X\to Y$ over $U$ satisfying the following.
\begin{itemize}
    \item For any integral curve $C$ such that the image of $C$ in $U$ is a point, then $\cont_R(C)$ is a point if and only if $[C]\in R$.
    \item $\Oo_Y=(\cont_R)_*\Oo_X$. In other words, $\cont_R$ is a contraction.
    \item Let $L$ be a line bundle on $X$ such that $L\cdot R=0$, then there exists a line bundle $L_Y$ on $Y$ such that $L=\cont_R^*L_Y$ .
\end{itemize}
\end{thm}

The author has been told by Jihao Liu that Theorem \ref{thm: Contraction thm for g-pairs} along with \cite[Theorem 1.1]{LX22b} would allow one to run MMP for glc pairs in the non-$\Qq$-factorial setting:

\begin{thm}\label{thm: mmp for glc pairs}
Let $(X,\Delta,\Mm)/U$ be a glc pair, then we can run a $(K_X+\Delta+\Mm_X)$-MMP over $U$.  
\end{thm}

This turns out to be useful for proving some expected semi-ampleness results since log bigness is not preserved when pulling back to $\Qq$-factorial gdlt models. Actually N. Tsakanikas and I will pursue the following statement in a forthcoming paper:

\begin{thm}\label{thm: gmm with A in the boundary}
Let $(X,\Delta,\Mm)/U$ be a glc pair and $A$ be an ample$/U$ $\Rr$-divisor such that $(X,\Delta+A,\Mm)$ is also glc. Then we can run a $(K_X+\Delta+\Mm_X+A)$-MMP which terminates with a Mori fiber space or a good minimal model (not necessarily $\Qq$-factorial).
\end{thm}


\vspace{.5em}

Since the proof of the main theorem relies on showing certain finiteness of relations, we will inevitably run into many technical issues, so we would like to give a sketch here to explain the core ideas in the proof.

\vspace{.5em}

\noindent\textit{ Sketch of the proof of Theorem \ref{thm: main thm 1.1}:} By perturbing the generalized pair and applying Fujino's technique (cf. \cite{Fuj11}), we can easily reduce the question to proving that $L|_V$ is semi-ample, where $V=\Ngklt(X,\Delta,\Mm)$ is the non-gklt locus of $(X,\Delta,\Mm)$ with the reduced scheme structure. The subtle thing here is that the structure of $V$ is somehow complicated (eg. $V$ may not be equi-dimensional or irreducible) so it is usually really hard to tell when a line bundle on $V$ should be semi-ample. 

However, $V$ is actually semi-normal and has a good stratification structure (glc stratification) coming from the glc centers (cf. \cite{LX22b}), and if we consider some nicely chosen stratified morphsim, for example, the normalization $\pi:V^n\to V$, then we can do subadjunction to $V^n$ and then by induction on the dimension we know that $L|_{V^n}$ is semi-ample. Notice that $V^n=\coprod V_i$ is a disjoint union of irreducible normal varieties, so for each $V_i$, $L|_{V_i}$ defines a contraction 
$$
g_i:V_i\to Z_i
$$
with a so called glc crepant log structure (see Definition \ref{defn: glc crepant log structure}), which induces a glc stratification on $Z_i$. In order to show that $L|_V$ is semi-ample, we must first find the correct candidate morphism $g:V\to Z$ that will be defined by $L|_V$ with the information coming from the $g_i$ and $\pi$. More precisely, we must consider the relation
$$
\{g_i(x)\sim g_j(y)~|~x\in V_i, y\in V_j, \pi(x)=\pi(y)\}
$$
between $Z_i$ and $Z_j$ ($i,j$ need not to be distinct). After some extra effort, we can give a nice interpretation on the induced relation between the $Z_i$'s by relating it with some group actions on the strata induced by the stratification. Fortunately for us, we have a powerful gluing theory introduced by Koll\'ar, with the help of which we only need to show that the above relation generated by $g_i$ and $\pi$ is finite in some sense (\cite[Theorem 9.21]{Kol13}). Moreover, we only need to check that this holds on each glc center $Z_{i,\gamma}\subset Z_i$. This is essentially equivalent to showing that the stablizer group stab$(Z_{i,\gamma})$ is finite.

The relation given by $\pi$ is always finite since $\pi$ is a finite morphism, but the contractions $g_i$ will create extra gluing information. A more careful computation shows that the extra relations essentially come from the different minimal glc centers on $V_i$ that dominate the same glc center on $Z_i$. For simplicity, we consider the case that there is only one $g_i:V_i\to Z_i$ and $g_i|_{V_{i,\alpha}}$ is also a contraction for any glc center $V_{i,\alpha}\subset V_i$. Let $V_{i,\alpha_1}$ and $V_{i,\alpha_2}$ be two different minimal glc center that dominate $Z_i$. Assume that the only gluing relation coming from $\pi$ is given by an isomorphism
$$
\tau_{12}:V_{i,\alpha_1}\to V_{i,\alpha_2}.
$$  
Then we can see that $\tau_{12}$ does not generate any automorphism in $V_{i,\alpha_1}$ or $V_{i,\alpha_2}$. However, $\tau_{12}$ induces an automorphism of $Z_i$, which may not be of finite order in general (see Example \ref{ex: relation is not finite in general} below). 

Nevertheless, the log bigness in our assumption makes the situation much better behaved (see Theorem \ref{thm: unique minimal glc center for log big glc crepant structure}). We actually show that for any glc center $Z_{i,\gamma}\subset Z_i$, there is a unique minimal glc center $V_{i,\alpha}$ that controls all the relations concerning $Z_{i,\gamma}$. In particular, any automorphism of $Z_{i,\gamma}$ coming from the relations will lift to an automorphism of $V_{i,\alpha}$ which in turn is induced by the $\pi$. Therefore the finiteness of relations between the $V_i$ will ensure the finiteness of relations between the $Z_i$ and we can obtain the geometric quotient $Z$ as we desired.

Applying the same philosophy to the total space of the line bundle $mL|_V$ over $V$ for sufficiently divisible $m$, we will be able to find a line bundle $H$ on $Z$ such that $g^*H=mL|_V$. We can easily show that $H$ is ample, hence $L|_V$ is semi-ample and we are done.

\vspace{1em}
The following example shows that the uniqueness of minimal glc centers is really necessary when using gluing theory to find the desired $Z$ and $H$ that correspond to the semi-ample $L$.

\begin{ex}[{\cite[Example 4.15]{LX22b}}]\label{ex: relation is not finite in general}
Let $\lambda\in\Cc^*$ and consider $\Pp^1\times \mathbb{A}^1$, which can be regarded as the total space of a trivial line bundle over $\Pp^1$. We define $\phi_\lambda: \{0\}\times\mathbb{A}^1\simeq\{\infty\}\times\mathbb{A}^1$ by $(0,t)\mapsto(\infty,\lambda t)$ and glue $\{0\}\times\mathbb{A}^1$ and $\{\infty\}\times\mathbb{A}^1$ together using $\phi_\lambda$ to get a demi-normal variety $M$ with projection $p:M\to C$, where $C$ is a nodal cubic. Then $M$ is the total space of a line bundle $N$ on $C$. Moreover, $N\in\Pic^0(C)\simeq\mathbb{G}_m=\Cc^*$ and can be canonically identified with $\lambda\in\Cc^*$. 

\begin{enumerate}
    \item Let $W:=\PP_C(\Oo_C\oplus N)$ be a $\Pp^1$-bundle over $C$, and let $C'\subset W$ be the section at infinity, which belongs to $|\Oo_W(1)|$. Then the normalization $W^n=\Pp^1\times\Pp^1$, and the extended isomorphism
    $$
    \tilde{\phi}_\lambda: \{0\}\times\Pp^1\simeq\{\infty\}\times\Pp^1~,~(0,[x,y])\mapsto(0,[\lambda x,y])
    $$
    gives the gluing relation of $\pi_W: W^n\to W$.
    
    Notice that $K_W$ is Cartier since $W$ is a locally complete intersection. Let $L:=K_X+3C'$, then we see that 
    $$
    \pi^*L=K_{W^n}+\{0\}\times\Pp^1+\{\infty\}\times\Pp^1+3\pi^*C'\sim p_2^*(\{\infty\})
    $$ is base point free and defines the second projection $p_2:W^n\to Y\simeq \Pp^1$. Since $V_1=\{0\}\times\Pp^1$ and $V_2=\{\infty\}\times\Pp^1$ both dominate $Y$ under $p_2$, $\tilde{\phi}_\lambda$ induces an automorphism of $V$:
    $$
    \tilde{\sigma}_\lambda:[x,y]\mapsto[\lambda x,y]
    $$
    Thus the relation generated by $\pi$ and $p_2$ is given by
$$\{[x,y]\sim[x',y']|[x',y']=[x,\lambda^ly] \text{ for some $l\in\Zz$}\}$$
and is finite if and only if $\lambda$ is a root of unity.
\item Let $\pi_C:\Pp^1\to C$ be the normalization. Then $\pi_C^*(N)\simeq\Oo_{\Pp^1}$ and it defines $g^n:\Pp^1\to\Spec~\Cc$, then the gluing relation $\{0\}\sim\{\infty\}$ from $\pi_C$ gives no extra relation under $g^n$, and so we get the morphism: $$
g: C\to\Spec~\Cc
$$ 

However, if we consider the total space $M$ then $\pi_M:\Pp^1\times\mathbb{A}^1\to M$ is the normalization. Notice that $\Pp^1\times\mathbb{A}^1$ is also the total space of the trivial line bundle over $\Pp^1$, thus there is a canonical morphism 
$$
g^n_M: \Pp^1\times\mathbb{A}^1\to \mathbb{A}^1
$$
between total spaces of corresponding line bundles coming from $g^{n,*}\Oo_{\Spec\Cc}=\Oo_{\Pp^1}$. Then $\phi_\lambda$ induces a automorphism of $\mathbb{A}^1$:
$$
\sigma_\lambda: t\mapsto\lambda t   
$$
Thus the relation generated by $\pi_M$ and $g^n_M$ on $\mathbb{A}^1$ is given by 
$$
\{t\sim s ~|~t=\lambda^l s \text{ for some $l\in\Zz$}\}
$$
and is finite if and only if $\lambda$ is a root of unity. 

Even if $\lambda$ is a root of unity, for example, assume $\lambda$ generates $\mu_{n}\subset\mathbb{G}_m$, we have $\mathbb{A}^1/\mu_n\simeq\mathbb{A}^1$ and get $g_M: M\to \mathbb{A}^1$ as our desired morphism. However, if we look at the equivariant $\mathbb{G}_m$ action under $g_M$, we see the action on $\mathbb{A}^1\backslash\{0\}$ is the natural $\mathbb{G}_m$ action on $\mathbb{G}_m/\mu_n$. This corresponds to the fact that $N$ is not a pullback of a line bundle on $\Spec~\Cc$. Actually the $\mathbb{A}^1\backslash\{0\}$ with the above $\mathbb{G}_m$ action is call a Seifert bundle (\cite[Definition 9.50]{Kol13}) and it will become a line bundle if we replace $N$ with $nN$.

\end{enumerate}

\end{ex}

\noindent\textbf{Acknowledgement}. The author would like to thank his advisor Christopher D. Hacon for useful discussions and constant support. He would like to thank Jihao Liu for introducing the questions and giving useful comments. He would also like to thank Jingjun Han and Nikolaos Tsakanikas for giving useful comments. The author is partially supported by NSF research grants no: DMS-1801851, DMS-1952522 and by a grant from the Simons Foundation; Award Number: 256202.

\section{Preliminaries}\label{sec: preliminaries}

We adopt the standard notation and definitions in \cite{KM98,BCHM10} and will freely use them. We will first introduce the definition of generalized pairs by using $\bb$-divisors. Then we will recall the glc crepant log structures and its induced glc stratifications developed in \cite{LX22b}. 

\subsection{Generalized pairs}

We will follow the original definitions in \cite{BZ16} but will adopt the same notations as in \cite{HL21a}. Notice that there are some small differences with the definitions in \cite{HL21a}: in this paper, all generalized (sub)-pairs are assumed to be NQC.

\begin{defn}[$\bb$-divisors]\label{defn: b divisors} Let $X$ be a normal quasi-projective variety. We call $Y$ a \emph{birational model} over $X$ if there exists a projective birational morphism $Y\to X$. 

Let $X\dashrightarrow X'$ be a birational map. For any valuation $\nu$ over $X$, we define $\nu_{X'}$ to be the center of $\nu$ on $X'$. A \emph{$\bb$-divisor} $\Dd$ over $X$ is a formal sum $\Dd=\sum_{\nu} r_{\nu}\nu$ where $\nu$ are valuations over $X$ and $r_{\nu}\in\mathbb R$, such that $\nu_X$ is not a divisor except for finitely many $\nu$. If in addition, $r_{\nu}\in\Qq$ for every $\nu$, then $\Dd$ is called a \emph{$\Qq$-$\bb$-divisor} over $X$. The \emph{trace} of $\Dd$ on $X'$ is the $\Rr$-divisor
$$\Dd_{X'}:=\sum_{\nu_{i,X'}\text{ is a divisor}}r_i\nu_{i,X'}.$$
If $\Dd_{X'}$ is $\Rr$-Cartier and $\Dd_{Y}$ is the pullback of $\Dd_{X'}$ on $Y$ for any birational model $Y$ of $X'$, we say that $\Dd$ \emph{descends} to $X'$ and write $\Dd=\overline{\Dd_{X'}}$. 

Let $X\rightarrow U$ be a projective morphism and assume that $\Dd$ is a $\bb$-divisor over $X$ such that $\Dd$ descends to some birational model $Y$ over $X$. If $\Dd_Y$ is nef$/U$ (resp. semi-ample$/U$), then we say that $\Dd$ is \emph{nef}$/U$ (resp. semi-ample$/U$). If $\Dd_Y$ is a Cartier divisor, then we say that $\Dd$ is \emph{$\bb$-Cartier}. If $\Dd$ can be written as an $\Rr_{\geq 0}$-linear combination of nef$/U$ $\bb$-Cartier $\bb$-divisors, then we say that $\Dd$ is \emph{NQC}$/U$.
\end{defn}

\begin{defn}[Generalized pairs]\label{defn: g-pairs}
A \emph{generalized sub-pair} (\emph{g-sub-pair} for short) $(X,\Delta,\Mm)/U$ consists of a normal quasi-projective variety $X$ associated with a projective morphism $X\rightarrow U$, an $\Rr$-divisor $\Delta$ on $X$, and an NQC$/U$ $\bb$-divisor $\Mm$ over $X$, such that $K_X+\Delta+\Mm_X$ is $\Rr$-Cartier. If $\Delta$ is a $\Qq$-divisor and $\Mm$ is a $\Qq$-$\bb$-divisor, then we say that $(X,\Delta,\Mm)/U$ is a \emph{$\Qq$-g-sub-pair}.



A g-sub-pair (resp. $\Qq$-g-sub-pair) $(X,\Delta,\Mm)/U$ is called a \emph{g-pair} (resp. \emph{$\Qq$-g-pair}) if $\Delta\geq 0$. A sub-pair $(X,\Delta)$ is called a \emph{pair} if $\Delta\geq 0$.
\end{defn}

\begin{defn}[Singularities of generalized pairs]\label{defn: sing of g-pairs}
	Let $(X,\Delta,\Mm)/U$ be a g-(sub-)pair. For any prime divisor $E$ and $\mathbb R$-divisor $D$ on $X$, we define $\mult_{E}D$ to be the \emph{multiplicity} of $E$ along $D$.  Let $h:W\to X$
	be any log resolution of $(X,\Supp\Delta)$ such that $\Mm$ descends to $W$, and let
	$$K_W+\Delta_W+\Mm_W:=h^*(K_X+\Delta+\Mm_X).$$
	The \emph{log discrepancy} of a prime divisor $D$ on $W$ with respect to $(X,\Delta,\Mm)$ is $1-\mult_{D}\Delta_W$ and it is denoted by $a(D,X,\Delta,\Mm).$
	
	We say that $(X,\Delta,\Mm)$ is \emph{(sub-)glc} (resp. \emph{(sub-)gklt}) if $a(D,X,\Delta,\Mm)\ge0$ (resp. $>0$) for every log resolution $h: W\to X$ as above and every prime divisor $D$ on $W$. We say that $(X,\Delta,\Mm)$ is \emph{gdlt} if $(X,\Delta,\Mm)$ is glc, and there exists a closed subset $V\subset X$, such that
\begin{enumerate}
    \item $X\backslash V$ is smooth and $\Delta_{X\backslash V}$ is simple normal crossing, and
    \item for any prime divisor $E$ over $X$ such that $a(E,X,\Delta,\Mm)=0$, $\Center_XE\not\subset V$ and $\Center_XE\backslash V$ is an lc center of $(X\backslash V,B|_{X\backslash V})$.
\end{enumerate}
	    
	 Suppose that $(X,\Delta,\Mm)$ is sub-glc. A \emph{glc place} of $(X,\Delta,\Mm)$ is a prime divisor $E$ over $X$ such that $a(E,X,\Delta,\Mm)=0$. A \emph{glc center} of $(X,\Delta,\Mm)$ is the center of a glc place of $(X,\Delta,\Mm)$ on $X$. The \emph{non-gklt locus} $\Ngklt(X,\Delta,\Mm)$ of $(X,\Delta,\Mm)$ is the union of all glc centers of $(X,\Delta,\Mm)$.
	 
 If a $\Qq$-g-pair $(X,\Delta,\Mm)$ is glc, then we will call $(X,\Delta,\Mm)$ a glc $\Qq$-pair for short.
 
 We say that an $\Rr$-Cartier divisor $D$ is {\it log big} over $U$ with respect to $(X,\Delta,\Mm)$ if $D$ is big over $U$ and
for any generalized lc center $T$ of $(X,\Delta,\Mm)/U$ with the normalization $T^n\to T$, the
pullback $D|_{T^n}$ is big over $U$.
\end{defn}

The following lemma is important for applying \cite[Theorem 13.1]{Fuj11} when we try to prove semi-ampleness by inductions. We refer the readers to \cite[Section 7]{Fuj11} for the definitions of non-lc ideal and locus.

\begin{lem}\label{lem: perturb glc pair to nlc pair}
Let $(X,\Delta,\Mm)/U$ be a glc pair, and $L$ a nef $\Qq$-divisor such that $L-(K_X+\Delta+\Mm_X)$ is nef and big$/U$. Then there exists an effective $\Rr$-divisor $B$ such that $L-K_X-B$ is ample over $U$ and $\Nlc(X,B)=\Ngklt(X,\Delta,\Mm)$ (as subschemes). 
\end{lem}
\begin{proof}
Let $f:Y\to X$ be a log resolution such that $\Exc(f)\cup f^{-1}\Delta$ is snc and $\Mm$ descends on $Y$. Let $K_Y+\Delta_Y+\Mm_Y=f^*(K_X+\Delta+\Mm_X)$, then $f^*L-(K_Y+\Delta_Y)+\Mm_Y$ is nef and big so there is an effective $\Rr$-divisor $E$ on $Y$ such that 
$$
f^*L-(K_Y+\Delta_Y)+\Mm_Y\sim_{\Rr,U} A_n+\frac{1}{n}E,
$$
where $A_n$ can be chosen to be a sufficiently general effective ample $\Rr$-divisor. By perturbing $A_n$ a little bit we can also assume that $\lf\Delta_Y\rf\subset\Supp E$. Since $\frac{1}{2}(L-K_X-\Delta-\Mm_X)$ is nef and big, there is an effective $\Rr$-divisor $E'$ on $X$ such that 
$$
\frac{1}{2}(L-K_X-\Delta-\Mm_X)\sim_{\Rr,U} A'_n+\frac{1}{n}E',
$$
where $A'_n$ is ample. Then $B:=f_*(\Delta_Y)+\frac{1}{2n}f_*(E+A_n)+\frac{1}{n}E'$ would work for $n\gg0$
\end{proof}

\subsection{Crepant log structures}

\begin{defn}\label{defn: glc crepant log structure}
A \emph{glc crepant log structure} is of the form $f: (X,\Delta,\Mm)\rightarrow Z$, where
\begin{enumerate}
    \item $(X,\Delta,\Mm)/Z$ is a glc g-pair,
    \item $K_X+\Delta+\Mm_X\sim_{\Rr,Z}0$, and
    \item $f$ is a contraction. In particular, $f_*\Oo_X=\Oo_Z$.
\end{enumerate}
In addition, if
\begin{enumerate}
    \item[(4)] $(X,\Delta,\Mm)$ is gdlt, 
\end{enumerate}
then we say that $f: (X,\Delta,\Mm)\rightarrow Z$ is a \emph{gdlt crepant log structure}.

For any irreducible subvariety $W\subset Z$, we say that $W$ is a \emph{glc center} of a glc crepant log structure $f: (X,\Delta,\Mm)\rightarrow Z$, if there exists a glc center $W_X$ of $(X,\Delta,\Mm)$ such that $W=f(W_X)$. For any (not necessarily closed) point $z\in Z$, we say that $z$ is a \emph{glc center} of $f: (X,\Delta,\Mm)\rightarrow Z$ if $\bar z$ is a glc center of $f: (X,\Delta,\Mm)\rightarrow Z$.
\end{defn}

Let us recall two important Lemmas in \cite{LX22b}:

\begin{lem}[{\cite[Lemma 3.17]{LX22b}}]\label{lem: glc locus is unibranch}
Let $f: (X,\Delta,\Mm)\rightarrow Z$ be a glc crepant log structure and $z\in Z$ a (not necessarily closed) point. Let $$\mathcal{S}_z:=\{V\mid V\text{ is a glc center of }f: (X,\Delta,\Mm)\rightarrow Z, z\in V\}.$$
Then:
\begin{enumerate}
    \item There exists a unique element $W\in\mathcal{S}_z$ that is minimal with respect to inclusion. 
    \item $W$ is unibranch at $z$, i.e.  the completion $\hat{W}_z$ is irreducible.
    \item Any intersection of glc centers of $f: (X,\Delta,\Mm)\rightarrow Z$ is also a union of glc centers.
\end{enumerate}
\end{lem}

\begin{lem}[{\cite[Lemma 3.19]{LX22b}}]\label{lem: gdlt crepant log structure is compatible under subadjunction}
Let $f: (X,\Delta,\Mm)\to Z$ be a gdlt crepant log structure and $Y\subset X$ a glc center. Let
$$
f|_Y: Y\xrightarrow{f_Y}Z_Y\xrightarrow{\pi} Z
$$
be the Stein factorization of $f|_Y$, and $(Y,\Delta_Y,\Mm^Y)/Z$ the gdlt g-pair induced by repeated adjunctions
$$K_Y+\Delta_Y+\Mm_Y^Y:=(K_X+\Delta+\Mm_X)|_Y.$$
Then:
\begin{enumerate}
\item $f_Y: (Y,\Delta_Y,\Mm^Y)\rightarrow Z_Y$ is a gdlt crepant log structure.
\item For any glc center $W_Y\subset Z_Y$ of $f_Y: (Y,\Delta_Y,\Mm^Y)\rightarrow Z_Y$, $\pi(W_Y)$ is a glc center of $f: (X,\Delta,\Mm)\rightarrow Z$.
\item For any glc center $W\subset Z$ of $f: (X,\Delta,\Mm)\rightarrow Z$, every irreducible component of $\pi^{-1}(W)$ is a glc center of  $f_Y: (Y,\Delta_Y,\Mm^Y)\rightarrow Z_Y$.
\end{enumerate}
\end{lem}

The following theorem is an analogue of \cite[Theorem-Definition 4.45]{Kol13}.

\begin{defthm}\label{thm: spring and source for glc crepant log structure}
Let $(X,\Delta,\Mm)\rightarrow Y$ be a gdlt crepant log structure and $Z\subset Y$ a glc center with normalization $n: Z^n\rightarrow Z$. Let $S\subset X$ be glc center of $(X,\Delta,\Mm)$ which dominates $Z$ and is minimal with respect to inclusion. Let $(S,\Delta_S,\Mm^S)$ be the gdlt g-pair induced by the adjunction
$$K_S+\Delta_S+\Mm^S_S:=(K_X+\Delta+\Mm_X)|_S$$
and let $f^n_S: S\rightarrow Z_S\rightarrow Z^n$ be the Stein factorization. Then: 
\begin{enumerate}
    
    \item (Uniqueness of springs) The isomorphism class of $Z_S$ does not depend on the choice of $S$. It is called the spring of $Z$ and denoted by $\Spr(Z,X,\Delta,\Mm)$, or $\Spr(Z,X)$ for short if there is no confusion.
    \item (Uniqueness of sources) The equivalence class of $S$ does not depend on the choice of $S$. 
    This birational class of $S$ will be called as the \emph{source of $Z$} and is denoted by $\Src(Z,X,\Delta,\Mm)$, or $\Src(Z,X)$ for short if there is no possible confusion.
    \item (Crepant log structure) $(S,\Delta_S,\Mm^S)$ is gdlt, $K_S+\Delta_S+\Mm^S_S\sim_{\Qq,Z}0$, and $(S,\Delta_S,\Mm^S)$ is gklt on the generic fiber of $f|_S$.
    \item (Galois property) The extension $Z_S\to Z$ is Galois and $\Bir_Z(S)\to\Gal(Z_S/Z)$ is surjective.
    \item (Adjunction) Let $W\subset X$ be a glc center, $f_W: W\to W_Y$ and $n_W:W_Y\to Y$ be the Stein factorization of $f|_W$. Let $Z_W\subset W_Y$ be an irreducible subvariety such that $n_W(Z_W)=Z$. Then $Z_W$ is a glc center of $(W,\Delta_W,\Mm^W)$ and 
    \begin{align*}
        \Src(Z,X,\Delta,\Mm)&\sim_{\Bir}\Src(Z_W,W,\Delta_W,\Mm^W)\\
        \Spr(Z,X,\Delta,\Mm)&\simeq\Spr(Z_W,W,\Delta_W,\Mm^W)
    \end{align*}
    
\end{enumerate}
\end{defthm}
\begin{proof}
By \cite[Theorem 3.16]{LX22b}(cf. \cite[Theorem 1.4]{FS20}), different choices of S are $\Pp^1$-linked to each other, hence they are birational, proving (1). This implies (2) while (3) is clear. For (5), note that $Z_W$ is a glc center by Lemma \ref{lem: gdlt crepant log structure is compatible under subadjunction} and we can actually choose representatives such that 
$$
\Src(Z_W,W,\Delta_W,\Mm^W)=\Src(Z,X,\Delta,\Mm),
$$
the rest then follows from (1) and (2). Finally, (4) follows from the proof of \cite[Lemma 4.46]{Kol13}
\end{proof}

\subsection{Generalized log canonical stratification}

We refer to \cite[Section 5, Section 9]{Kol13} and \cite[Section 4]{LX22b} for more details. 


\begin{defn}[{\cite[Definition 9.15]{Kol13}}] 
Let $X$ be a scheme. A {\it stratification} of $X$ is a decomposition of $X$ into a finite disjoint union of reduced locally closed subschemes. We will consider stratifications where the strata are of pure dimensions
and are indexed by their dimensions. We write $X=\cup_{i}S_iX$ where $S_iX\subset X$ is the $i$-th
dimensional stratum. Such a stratified scheme is denoted by $(X,S_*)$. We also
assume that $\cup_{i\le j}S_iX$ is closed for every $j$. The {\it boundary} of $(X,S_*)$ is the closed subscheme
$$
B(X,S_*):=\cup_{i<\dim X}S_iX=X\backslash S_{\dim X}X,
$$
and is denoted by $B(X)$ if the stratification $S_*$ is clear. We call $S_{\dim X}$ the {\it open stratum}.

Let $(X, S_*)$ and $(Y, S_*)$ be stratified schemes. We say that $f:X\to Y$ is a {\it stratified morphism} if $f(S_iX)\subset S_iY$ for every $i$. Since $S_iX$ are disjoint with each other, $f: X\to Y$ is a stratified morphism if and only if $S_iX=f^{-1}(S_iY)$.

Let $(Y, S_*)$ be a stratified scheme and $f:X\to Y$ a quasi-finite morphism such that $f^{-1} (S_iY)$ has pure dimension $i$ for every $i$ . Then $S_iX:=f^{-1}(S_iY)$ defines a stratification of $X$. We denote it by $(X,f^{-1}S_*)$, and we say that $f:X\to(Y,S_*)$ is \emph{stratifiable}.
\end{defn}

\begin{defn}
Let $(X, S_*)$ be a stratified variety. A relation $(\sigma_1,\sigma_2): R\rightrightarrows (X,S_*)$ is {\it stratified} if each $\sigma_i$ is stratifiable and $\sigma_1^{-1}S_*=\sigma_2^{-1}S_*$. Equivalently,
there exists a stratification $(R,\sigma^{-1}S_i)$, such that $r\in\sigma^{-1}S_iR$ if and only if $\sigma_1(r)\in S_iX$ and if and only if $\sigma_2(r)\in S_iX$.
\end{defn}

Next we give a special stratification that is induced by the glc crepant log structure. 

\begin{defn}[Glc stratification]
Let $f:(X,\Delta,\Mm)\to Z$ be a glc crepant log structure. Let $S^*_i(Z,X,\Delta,\Mm)\subset Z$ be the union of all $\le i$-dimensional glc centers of $f:(X,\Delta,\Mm)\to Z$, and
$$
S_i(Z,X,\Delta,\Mm):=S^*_i(Z,X,\Delta,\Mm)~\backslash ~S^*_{i-1}(Z,X,\Delta,\Mm).
$$
If the glc crepant log structure $f:(X,\Delta,\Mm)\to Z$ is clear from the context, we will use $S_i(Z)$ for abbreviation. It is clear that each $S_i(Z)$ is a locally closed subspace of $Z$ of pure dimension $i$, and $Z$ is the disjoint union of all $S_i(Z)$. 

The stratification of $Z$ induced by $S_i(Z)$ is called the \emph{generalized log canonical stratification} (\emph{glc stratification} for short) of $Z$ induced by $f:(X,\Delta,\Mm)\to Z$.  Since this is the only stratification we are going to use in the rest of this paper, we usually will not emphasize the glc crepant structure $f:(X,\Delta,\Mm)\to Z$, and we will denote the corresponding stratified scheme by $(Z,S_*)$. The \emph{boundary} of $(Z,S_*)$ is the closed subspace
$$B(Z,S_*):=Z\backslash S_{\dim Z}(Z)=\cup_{i<\dim Z}S_i(Z).$$
\end{defn}

\section{Crepant log structure with log bigness}

In this section we show that there will be no $\Pp^1$-link in a glc crepant log structure if some mild log bigness assumptions are posed on the g-pair.

\begin{lem}\label{lem: connectedness principle and compatibility for glc crepant structure}
Let $f:(X,\Delta,\Mm)\to Z$ be a glc crepant log structure and $V$ be a glc center on $Z$. Let $W$ and $W'$ be two minimal glc centers on $X$ that dominates $V$. Let $\tilde{W}\subset X$ be another glc center such that $V\subset f(\tilde{W})$ (We allow $X$ itself to be a glc center when $Z=V$). Then the followings hold:
\begin{enumerate}
    \item There exist glc centers $W_0,W_1\cdots,W_n$ and $\hat{W}_1,\cdots,\hat{W}_n$ on $X$ such that $W=W_0\subset\hat{W}_1\supset W_1\subset\cdots\subset\hat{W}_n\supset W_n=W'$ and $f(W_i)=f(\hat{W}_i)=V$.
    \item There exists a glc center $W''\subset X$ such that $W''\subset\tilde{W}$ and $f(W'')=V$, hence we can also choose such $W''$ to be minimal.
\end{enumerate}
\end{lem}
\begin{proof}
If $(X,\Delta,\Mm)$ is $\Qq$-factorial gdlt, then (1) follows directly from \cite[Theorem 3.16]{LX22b} (cf. \cite[Theorem 1.4]{FS20}) and (2) follows from Lemma \ref{lem: gdlt crepant log structure is compatible under subadjunction} (3). In the general case we can take a gdlt modification $g:(Y,\Delta_Y,\Mm)\to(X,\Delta,\Mm)$ and replace $\tilde{W}$ (resp. $W$, $W'$) by any (resp. minimal) glc centers on $Y$ that dominates $W$ (resp. $W$, $W'$). Then the statements follow by looking at the images of those glc centers under $g$.
\end{proof}

\begin{thm}\label{thm: unique minimal glc center for log big glc crepant structure}
Let $(X,\Delta,\Mm)/U$ be a glc pair, and $L$ a $\Qq$-Cartier $\Qq$ divisor such that $A:=L-(K_X+\Delta+\Mm_X)$ is nef and log big$/U$ with respect to $(X,\Delta,\Mm)$. Assume that $L$ is semi-ample and defines a projective contraction $\phi: X\to Z$ over $U$. Then $\phi$ can also be regarded as a glc crepant log structure $(X,\Delta,\bar{A}+\Mm)\to Z$. Let $V$ be any glc center on $Z$, then there exists a unique minimal glc center $W$ on $X$ such that
\begin{itemize}
    \item $W$ dominates $V$, or in other words, $\phi(W)=V$.
    \item For any glc center $\tilde{W}$ on $X$ such that $V\subset\phi(\tilde{W})$, we have $W\subset \tilde{W}$.
\end{itemize}
Moreover, let $\tilde{W}\subset X$ be any glc center that dominates $Z$ and $\tilde{W}^n$ be the normalization, then $\tilde{W}^n\to Z$ is a contraction, or equivalently, the Stein factorization of $\tilde{W}^n\to Z$ is trivial.
\end{thm}

\begin{proof}
We use induction on the dimension. If $\dim X=1$, then $\phi$ is an isomorphism unless $X=\Pp^1$, for which case the statements are also straightforward.

By shrinking $Z$ we can assume that $V$ is the unique minimal glc center on $Z$. Let $W\subset X$ be a minimal glc center over $Z$, which implies $\phi(W)=V$ by our assumption, then it suffices to prove that any other glc centers $\tilde{W}\subset X$ intersects with $W$, which is equivalent to $W\subset\tilde{W}$ since $W$ is minimal. By Lemma \ref{lem: connectedness principle and compatibility for glc crepant structure} (2) we only need to show this for those minimal $\tilde{W}$ that $\phi(\tilde{W})=V$. We first claim that there is no glc centers $\tilde{W}$ and $\hat{W}$ such that
\begin{itemize}
    \item $\phi(\tilde{W})=\phi(\hat{W})=V$.
    \item $\tilde{W}$ intersects with $\hat{W}$.
    \item $W\subset \hat{W}$ but $\tilde{W}\cap W=\emptyset$.
\end{itemize}
Indeed, if there is such a $\tilde{W}$, then by Lemma \ref{lem: glc locus is unibranch} we can assume $\tilde{W}\subset\hat{W}$. Let $\hat{W}^n$ be the normalization of $\hat{W}$, then by \cite[Theorem 1.2]{HL21b} there is a glc structure $(\hat{W}^n,\Delta_{\hat{W}^n},\Mm^{\hat{W}^n})$ induced by the subadjunction 
$$
(K_X+\Delta+\Mm_X)|_{\hat{W}^n}=K_{\hat{W}^n}+\Delta_{\hat{W}^n}+\Mm_{\hat{W}^n}^{\hat{W}^n}
$$
and the glc centers on $\hat{W}^n$ exactly comes from the pullbacks of glc centers contained in $\hat{W}$. Let $\hat{W}^n\to V'\to V$ be the Stein factorization of $\hat{W}^n\to V$, then the glc centers dominating $V'$ are exactly those dominating $V$. Therefore by the induction hypothesis $\tilde{W}$ should contains $W$ (after pullback to $\hat{W}^n$), which is a contradiction. 

Now by Lemma \ref{lem: connectedness principle and compatibility for glc crepant structure} (1) we have glc centers $W_0,W_1\cdots,W_n$ and $\hat{W}_1,\cdots,\hat{W}_n$ on $X$ such that $W=W_0\subset\hat{W}_1\supset W_1\subset\cdots\subset\hat{W}_n\supset W_n=W'$ and $f(W_i)=f(\hat{W}_i)=V$. Thus by keep using the claim above we see that $W_i$ and $\hat{W}_i$ all contain $W$.

Moreover, we can consider
$$
K_F+\Delta_F+\Mm^F_F=(K_X+\Delta+\Mm_X)|_F
$$
for the general fiber $F$ of $\phi$ and apply the above statements to $(F,\Delta_F,\Mm^F)$. Then we can see that $\tilde{W}|_F$ is connected since any irreducible component of it is a glc center, which will contain a unique common minimal glc center. Therefore $\phi|_{\tilde{W}^n}: \tilde{W}^n\to Z$ must be a contraction.
\end{proof}

Actually Theorem \ref{thm: unique minimal glc center for log big glc crepant structure} holds in a much more general setting:

\begin{thm}
Let $(X,\Delta_1,\Mm^1)/U$ and $(X,\Delta_2,\Mm^2)/U$ be two glc $\Qq$-pairs with exactly the same glc places (centers). Assume that $(K_X+\Delta_1+\Mm^1_X)-(K_X+\Delta_2+\Mm^2_X)$ is log big$/U$ with respect to $(X,\Delta_1,\Mm^1)$. Furthermore, $(K_X+\Delta_1+\Mm^1_X)$ is semi-ample over $U$ and defines a glc crepant log structure $f:X\to Z$. Let $V$ be any glc center on $Z$, then there exists a unique minimal glc center $W$ on $X$ such that
\begin{itemize}
    \item $W$ dominates $V$, or in other words, $\phi(W)=V$.
    \item For any glc center $\tilde{W}$ on $X$ such that $V\subset\phi(\tilde{W})$, we have $W\subset \tilde{W}$.
\end{itemize}
Moreover, let $\tilde{W}\subset X$ be any glc center that dominates $Z$ and $\tilde{W}^n$ be the normalization, then $\tilde{W}^n\to Z$ is a contraction, or equivalently, the Stein factorization of $\tilde{W}^n\to Z$ is trivial.
\end{thm}

\begin{proof}
The assumptions preserve under adjunctions, so we can use induction on dimension and the proof now is exactly the same as that of Theorem \ref{thm: unique minimal glc center for log big glc crepant structure}.
\end{proof}

\section{Gluing relations on glc crepant log structures}

Before giving the proof of Theorem \ref{thm: main thm 1.1}, we need to make some preparations in order to describe the relations on those glc centers in a clear way. We will keep using the notations in the whole section.

\vspace{.5em}

Let $(X,\Delta,\Mm)/U$ be a glc pair and let $f:(Y,\Delta_Y,\Mm)\to(X,\Delta,\Mm)$ be a $\Qq$-factorial gdlt modification. Let $(X,S_*)$ and $(Y,S_*)$ be the natural glc stratifications. Let $W:=\lf\Delta_Y\rf$ be the boundary $B(Y)$ of $(Y,S_*)$. Then $W$ is demi-normal and let $\pi: W^n\to W$ be the normalization. Let $D$ be the double-normal locus on $W^n$, $D^n$ the normalization of $D$ and $\tilde{\tau}:D^n\to D^n$ the induced involution. We see that $\tilde{\tau}$ induces a gluing relation $R(\tilde{\tau})$ on $(W^n,S_*)$, this relation is finite and $W^n/R(\tilde{\tau})=W$.

Write $W^n=\coprod W_i$, where $W_i$ is irreducible for any $i$. Let $W_i\stackrel{f_i}{\longrightarrow} V_i\stackrel{n_i}{\longrightarrow}X$ be the Stein factorization of $f|_{W_i}$. Then we have the contraction:
$$
\amalg f_i: \coprod W_i \to \coprod V_i
$$

For any glc center $V_{i,\alpha}\subset V_i$, we define $\tilde{V}_{i,\alpha}:=\Spr(V_{i,\alpha},W_i)$ and let $h_{i,\alpha}:\tilde{V}_{i,\alpha}\to V_{i,\alpha}$ be the corresponding finite morphism, which is stratified by Lemma \ref{lem: gdlt crepant log structure is compatible under subadjunction}.

Let $W_{i,\alpha,k}\subset W_i$ be a minimal glc center that dominates $V_{i,\alpha}$ and if there is a gluing relation 
$$
\tilde{\tau}_{i,\alpha,k,T}:W_{i,\alpha,k}\to T
$$
where $T$ is a glc center of $W_j$ for some $j$. Let $f_j(T)=V_{j,\beta}$, since we have the following commutative diagram:

\begin{displaymath}
    \xymatrix{
        W_{i,\alpha,k} \ar@{^(->}[dr]\ar@{->}[d]_{\tilde{\tau}_{i,\alpha,k,T}} \\
        T \ar@{^(->}[r] & Y
    }
\end{displaymath}
We see that $n_j|_{V_{j,\beta}}\circ f_j|_{T}=f|_{T},~n_i|_{V_{i,\alpha}}\circ f_i|_{W_{i,\alpha,k}}=f|_{W_{i,\alpha,k}}$ and $f|_{W_{i,\alpha,k}}=f|_T\circ\tilde{\tau}_{i,\alpha,k,T}$, thus $T$ is also a minimal glc center that dominates $V_{j,\beta}$ since $W_{i,\alpha,k}$ is a minimal glc center that dominates $V_{i,\alpha}$. Hence $\tilde{\tau}_{i,\alpha,k,T}$ induces an isomorphism $\tau_{i,\alpha,k,T}: \tilde{V}_{i,\alpha}\to\tilde{V}_{j,\beta}$ as follows: 
\begin{displaymath}
    \xymatrix{
        W_{i,\alpha,k} \ar[r]\ar[d]_{\tilde{\tau}_{i,\alpha,k,T}}& \tilde{V}_{i,\alpha}\ar[r]^{h_{i,\alpha}}\ar[d]_{\tau_{i,\alpha,k,T}} & V_{i,\alpha} \\
        T \ar[r]             & \tilde{V}_{j,\beta}\ar[r]^{h_{j,\beta}}  & V_{j,\beta}
    }
\end{displaymath}
Notice that $\tau_{i,\alpha,k,T}$ is isomorphic on glc centers and gives a relation $\{h_{i,\alpha}(x)\sim h_{j,\beta}(\tau_{i,\alpha,k,T}(x))~|~x\in \tilde{V}_{i,\alpha}\}$ between $V_{i,\alpha}$ and $V_{j,\beta}$, which is the same as 
$$
\{f_i(x)\sim f_j(\tilde{\tau}_{i,\alpha,k,T}(x))~|~x\in W_{i,\alpha,k}\}.
$$

If $T'\subset W_i$ is another glc center such that $V_{i,\alpha}\subset f_i(T')$ with a gluing relation $\tilde{\tau}_{T',T''}:T'\to T''$, where $T''$ is a glc center on $W_j$, then by Lemma \ref{lem: connectedness principle and compatibility for glc crepant structure}(2) we know that there is an minimal glc center $W_{i,\alpha,k}\subset T'$ that dominates $V_{i,\alpha}$. Since $\tilde{\tau}_{T',T''}$ is isomorphic on glc centers, we get a gluing relation $\tilde{\tau}_{i,\alpha,k,T}:W_{i,\alpha,k}\to T$ by restriction. Let $f_j(T)=V_{j,\beta}$ Then we can easily see that the relation
$$
\{f_i(x)\sim f_j(\tilde{\tau}_{T',T''}(x))~|~x\in f_i^{-1}(V_{i,\alpha})\cap T'\}
$$
between $V_{i,\alpha}$ and $V_{j,\beta}$ is the same as the one given by the $\tau_{i,\alpha,k,T}$ 
above. 

Therefore it suffices to consider all possible $\tau_{i,\alpha,k,T}$ and let $R(\tau)$ be the induced pro-finite relation on $\coprod V_i$. Then $R(\tau)$ will reflect the relation generated by $\tilde{\tau}$ and $\coprod f_i$. Notice that $W^n/R(\tilde{\tau})=W$, and all $f_i$ come from $f$, the relation actually equals to
$$
\{f_i(x)\sim f_j(y)~|~x\in W_i, y\in W_j, \pi(x)=\pi(y)\}
$$
Since $-(K_Y+\Delta_Y+\Mm_Y)$ is nef and big over $X$, we also have $f_*\Oo_{W}=\Oo_{f(W)}$. Then it is not hard to see the pro-finite relation $R(\tau)$ is actually a finite relation (by looking at the geometric points of a fiber) and $f(W)=\coprod V_i/R(\tau)$ is the geometric quotient (see for example \cite[Propostion 3.12]{HX13}).

\begin{thm}\label{thm: semi-ampleness for glc crepant log structure with log bigness}
Let $(X,\Delta,\Mm)/U$ be a glc $\Qq$-pair, and $L$ a nef $\Qq$-divisor such that $A:=L-(K_X+\Delta+\Mm_X)$ is nef and log big$/U$ with respect to $(X,\Delta,\Mm)$. Then $L$ is semi-ample over $U$. 
\end{thm}

\begin{proof}
There is a natural glc stratification $(X,S_*)$ on $X$. Then by Lemma \ref{lem: perturb glc pair to nlc pair} and \cite[Theorem 13.1]{Fuj11} it suffices to prove that $L|_{B(X)}$ is semi-ample. Now we consider the new g-pair $(X,\Delta,\bar{A}+\Mm)$, which induces the same glc stratification, and fix a $\Qq$-factorial gdlt modification $f:(Y,\Delta_Y,\bar{A}+\Mm)\to(X,\Delta,\bar{A}+\Mm)$. 

We use the notations above, by induction on the dimension we can assume that $L|_{V_i}$ is semi-ample and defines a contraction $g_i:V_i\to Z_i$, so we have
$$
\amalg g_i: \coprod V_i\to\coprod Z_i.
$$

Notice that the composition $W_i\to V_i\to Z_i$ is a gdlt crepant log structure induced by $K_Y+\Delta_Y+\Mm_Y+f^*A$. Then for any glc center $Z_{i,\gamma}\subset Z_i$, we define $\tilde{Z}_{i,\gamma}:=\Spr(Z_{i,\gamma},W_i)$ and let $p_{i,\gamma}:\tilde{Z}_{i,\gamma}\to Z_{i,\gamma}$ be the canonical finite stratified morphism. 

By Theorem \ref{thm: unique minimal glc center for log big glc crepant structure}, there exists a unique minimal glc center $V_{i,\alpha}\subset V_i$ that dominates $Z_{i,\gamma}$. And we have the following diagram:
\begin{displaymath}
    \xymatrix{
        \tilde{V}_{i,\alpha}\ar[r]^{h_{i,\alpha}}\ar[d]_{\tilde{g}_{i,\gamma}}& V_{i,\alpha}\ar@{^(->}[r]\ar[d]& V_i\ar[d]_{g_i} \\
        \tilde{Z}_{i,\gamma} \ar[r]^{p_{i,\gamma}}& Z_{i,\gamma}\ar@{^(->}[r] & Z_i
    }
\end{displaymath}
Then it is easy to see that $\tilde{g}_{i,\gamma}$ is a contraction by considering the gdlt crepant log structure $W_i\to Z_i$. Assume that there is a gluing relation $\tau_{i,\alpha,k,T}:\tilde{V}_{i,\alpha}\to\tilde{V}_{j,\beta}$. Since the pullback of $L|_W$ from $W$ is indeed $\tau_{i,\alpha,k,T}$ invariant, $V_{j,\beta}\subset V_j$ is also the unique minimal glc center that dominates $Z_{j,\theta}=g_j(V_{j,\beta})$. Hence we have a induced gluing relation $\sigma_{i,\alpha,k,T}:\tilde{Z}_{i,\gamma}\to\tilde{Z}_{j,\theta}$ as follows:
\begin{displaymath}
    \xymatrix{
        W_{i,\alpha,k} \ar[r]\ar[d]_{\tilde{\tau}_{i,\alpha,k,T}}& \tilde{V}_{i,\alpha}\ar[r]\ar[d]_{\tau_{i,\alpha,k,T}} & \tilde{Z}_{i,\gamma}\ar[d]_{\sigma_{i,\alpha,k,T}}\ar[r]^{p_{i,\gamma}}&Z_{i,\gamma} \\
        S \ar[r]             & \tilde{V}_{j,\beta}\ar[r] & \tilde{Z}_{j,\theta}\ar[r]^{p_{j,\theta}}& Z_{j,\theta}
    }
\end{displaymath}

This gives a relation 
$$
\{p_{i,\gamma}(x)\sim p_{j,\theta}(\sigma_{i,\alpha,k,T}(x))~|~x\in\tilde{Z}_{i,\gamma}\}
$$
between $Z_{i,\gamma}$ and $Z_{j,\theta}$. As we have discussed above, the relation induced by $\coprod g_i$ and $R(\tau)$ is generated by all the $\sigma_{i,\alpha,k,T}$, which we use $R(\sigma)$ to denote. Now we want to show that $R(\sigma)$ is a finite relation. It suffices to check on each open stratum of $\tilde{Z}_{i,\gamma}$. Notice that any relation on the open stratum of $\tilde{Z}_{i,\gamma}$ itself is induced by the following form
$$
\sigma_{i_0,\gamma_0,i_1,\gamma_1}\circ\cdots\circ\sigma_{i_{n-1},\gamma_{n-1},i_n,\gamma_n},
$$
where 
$\sigma_{i_l,\gamma_{l},i_{l+1},\gamma_{l+1}}:\tilde{Z}_{i_l,\gamma_l}\to\tilde{Z}_{i_{l+1},\gamma_{l+1}}$ is the isomorphism induced by some $\tau_{i,\alpha,k,T}$ in $R(\tau)$ and $(i_0,\gamma_0)=(i_n,\gamma_n)=(i,\gamma)$. Then by Theorem \ref{thm: unique minimal glc center for log big glc crepant structure} for each $Z_{i_l,\gamma_l}$, there is a unique minimal glc center $V_{i_l,\alpha_l}\subset V_{i_l}$ that dominates $Z_{i_l,\gamma_l}$. Therefore this relation lifts to a relation $\tau_{i_0,\alpha_0,i_1,\alpha_1}\circ\cdots\circ\tau_{i_{n-1},\alpha_{n-1},i_n,\alpha_n}$ in $R(\tau)$:
\begin{displaymath}
    \xymatrix{
        \tilde{V}_{i,\alpha}\ar@{=}[r]&\tilde{V}_{i_0,\alpha_0} \ar[rr]^{\tau_{i_0,\alpha_0,i_1,\alpha_1}}\ar[d]& & \tilde{V}_{i_1,\alpha_1}\ar[r]|\cdots\ar[d] & \tilde{V}_{i_n,\alpha_n}\ar[d]\ar@{=}[r]&\tilde{V}_{i,\alpha} \\
        \tilde{Z}_{i,\gamma}\ar@{=}[r]&\tilde{Z}_{i_0,\gamma_0} \ar[rr]^{\sigma_{i_0,\gamma_0,i_1,\gamma_1}}& & \tilde{Z}_{i_1,\gamma_1}\ar[r]|\cdots & \tilde{Z}_{i_n,\gamma_n}\ar@{=}[r]&\tilde{Z}_{i,\gamma} 
    }
\end{displaymath}
Since the relation $R(\tau)$ is finite, so does $R(\tau)|_{\tilde{V}_{i,\alpha}}$, thus the relation $R(\sigma)$ on the open stratum of $\tilde{Z}_{i,\gamma}$ is also finite. Then by the construction we can see that $R(\sigma)$ is actually a finite, set theoretic, stratified equivalence relation (cf. \cite[Lemma 4.14]{LX22b}). Therefore by \cite[Theorem 9.21]{Kol13} we know that the geometric quotient $Z:=\coprod Z_i/R(\sigma)$ exists and there is a morphism $g:f(W)\to Z$ over $U$. 

Let $m$ be sufficiently divisible such that $mL$ is Cartier and $M_i:=mL|_{V_i}$ defines $g_i:V_i\to Z_i$ for each $i$. Then as in \cite[Construction 4.13]{LX22b}, we can consider the total spaces $V_{M_i}$ of the line bundles $M_i$ over $V_i$, and the total spaces of $Z_{H_i}$ of the line bundles $H_i$ over $Z_i$, where $H_i$ is very ample and $g^*_iH_i=M_i$. Similarly, we can define the corresponding relation $R(\tau_M)$ on $\coprod V_{M_i}$ and the induced $R(\sigma_H)$ on $\coprod Z_{H_i}$. Then by applying the same statements above we can deduce that $R(\sigma_H)$ is also a finite, stratified equivalence relation. Possibly by replacing $m$ with a multiple, the geometric quotient $\coprod Z_{H_i}/R(\sigma_H)$ exists by \cite[Theorem 9.21]{Kol13} and is the total space of an ample line bundle $H_Z$ over $Z$, where $g^*H_Z=mL|_{f(W)}$. Therefore $L|_{f(W)}$ is semi-ample over $U$ and we are done.
\end{proof}

\begin{proof}[Proof of Theorem \ref{thm: bpf thm for g-pairs}]
We use induction on the dimension.
By Lemma \ref{lem: perturb glc pair to nlc pair} and \cite[Theorem 13.1]{Fuj11} (cf. \cite[Theorem 7.2]{Amb03}) we only need to show that $\Oo_V(mL)$ is globally generated over $U$ for $m\gg0$, where $V=\Ngklt(X,\Delta,\Mm)$. We will use the notations in the proof of Theorem \ref{thm: semi-ampleness for glc crepant log structure with log bigness}, by which we know that $L$ is semi-ample over $U$. 

By the induction hypothesis, we know that for any $m\gg0$, $mL|_{V_i}$ induces the contraction $g_i:V_i\to Z_i$, thus $L_i:=L|_{V_i}=g_i^*K_i$ for some ample line bundle on $Z_i$. Let $V_{L_i}$ (resp. $Z_{K_i}$) be total space of the line bundle $L_i$ (resp. $K_i$) over $V_i$ (resp. $Z_i$), and $R(\tau_L)$ (resp. $R(\sigma_K)$) the corresponding gluing relation. Then $\coprod V_{L_i}/R(\tau_L)$ is just the total space of the line bundle $L|_{V}$ over $V$. Since all the relation $R(\sigma_K)$ could lift to a relation on $R(\tau_L)$, it implies that $\coprod Z_{K_i}/R(\sigma_K)$ is also the total space of an ample line bundle $H_Z$ over $Z$, where $g^*H_Z=L|_V$. Therefore $mH_Z$ is very ample for $m\gg0$ by Serre Vanishing and Castelnuovo–Mumford regularity (cf. \cite[1.8.22]{Laz04}), which shows that $mL|_V=g^*(mH_Z)$ is globally generated for all $m\gg0$. 
\end{proof}

\begin{proof}[Proof of Theorem \ref{thm: K_X+B+M+A is semi-ample for A ample}]
Since ampleness is an open condition, we can assume that $(X,\Delta,\Mm)$ is a glc $\Qq$-pair by using Shokurov-type rational polytopes (cf. \cite[Proposition 3.16]{HL22} and \cite[Lemma 5.3]{HLS19}). Then the statement follows by Theorem \ref{thm: semi-ampleness for glc crepant log structure with log bigness}.
\end{proof}

\begin{proof}[Proof of Theorem \ref{thm: semi-ampleness with good B+(A)}]
By the definition of $\mathbf{B}_+(A/U)$ (\cite[Definition 3.5.1]{BCHM10}), we can write 
$$
A\sim_{\Rr,U}H+E
$$
such that $H$ is ample and $E\ge0$ contains no glc center of $(X,\Delta+A,\Mm)$. Then for any sufficiently small $\epsilon>0$, $(X,\Delta+A+\epsilon E,\Mm)$ is still glc and its glc centers are the same as those of $(X,\Delta+A,\Mm)$. Therefore we can let $L=(K_X+\Delta+A+\Mm_X)$ and consider the glc pair 
$$
(X,\Delta+(1-\epsilon)A+\epsilon E,\Mm),
$$
then the statement follows from Theorem \ref{thm: K_X+B+M+A is semi-ample for A ample}.
\end{proof}

\begin{proof}[Proof of Theorem \ref{thm: Contraction thm for g-pairs}]
By using Shokurov-type rational polytopes (cf. \cite[Proposition 3.16]{HL22} and \cite[Lemma 5.3]{HLS19}) we can assume that $(X,\Delta,\Mm)$ is a glc $\Qq$-pair. Since $R$ is an $(K_X+\Delta+\Mm_X)$-negative extremal ray, there exists an ample $\Qq$-divisor $A$ such that $H:=K_X+\Delta+\Mm_X+A$ is nef and $R=\overline{\NE}(X/U)\cap H^\perp$. Then by Theorem \ref{thm: semi-ampleness for glc crepant log structure with log bigness} $H$ is semi-ample and defines a contraction $\varphi_H:X\to Z$ over $U$. Let $L$ be a line bundle on $X$ such that $L\cdot R=0$. Then both $H-(K_X+\Delta+\Mm_X)$ and $L+H-(K_X+\Delta+\Mm_X)$ is ample over $Z$. Hence by Theorem \ref{thm: bpf thm for g-pairs} $\Oo_X(mL+mH)$ and $\Oo_X(mH)$ are both globally generated over $Z$ for $m\gg0$. Then it is not hard to see that 
$$
\Oo_X(mL+mH)\sim_{\varphi_H}\Oo_X(mH)\sim_{\varphi_H}\Oo_X
$$ 
for all $m\gg0$, so $\Oo_X(mL)\sim_{\varphi_H}\Oo_X$ for all $m\gg0$ and therefore 
$$
\Oo_X(L)=\Oo_X((m+1)L)\otimes\Oo_X(-mL)
$$ is a pullback of a line bundle on $Z$.
\end{proof}

\begin{proof}[Proof of Theorem \ref{thm: mmp for glc pairs}]
If $K_X+\Delta+\Mm_X$ is not nef, then by \cite[Theorem 1.3]{HL21a} there exists a $(K_X+\Delta+\Mm)$-negative extremal ray $R$ in $\overline{\NE}(X/U)$. Now by Theorem \ref{thm: Contraction thm for g-pairs} we have a projective contraction $f:X\to Z$ contracting $R$. 

If $\dim X>\dim Z$, then $f:X\to Z$ is a Mori fiber space and we are done. If $f$ is birational then by \cite[Theorem 1.1]{LX22b} we know there exist a canonical model 
$$
X':=\proj_{Z}\oplus_{m\ge0}f_*\Oo_{X}(m(K_X+\Delta+\Mm_X))
$$
over $Z$. In particular, $K_{X'}+\Delta'+\Mm_{X'}$ is $\Rr$-Cartier, where $\Delta'$ is the birational transform. Then we can do the same thing to $(X',\Delta',\Mm)$ and keep running the MMP (see \cite[Section 4.9]{Fuj17} for more details).

\end{proof}

\end{document}